\documentclass[a4paper,10pt]{amsart}

\usepackage[utf8]{inputenc}
\usepackage{url}

\usepackage[expansion=false]{microtype}
\usepackage{amssymb,amsmath,amsopn,amsxtra,amsthm,amsfonts}
\usepackage{xr}
\usepackage[mathcal]{euscript}
\usepackage{mathalfa}

\usepackage[english, status=final, nomargin, inline]{fixme}
\fxusetheme{color}

\usepackage[pdfusetitle,unicode,hidelinks]{hyperref}

\usepackage{enumitem}
\setlist[enumerate,1]{label={\normalfont(\roman*)},itemsep=\parskip} 
\setlist[itemize,1]{itemsep=\parskip} 
\newlist{thmlist}{enumerate}{2}
\setlist[thmlist,1]{label={\em(\roman*)},ref={(\roman*)},%
  itemsep=\parskip,leftmargin=*,align=left}
\setlist[thmlist,2]{label={\em(\alph*)},ref={(\alph*)},%
  itemsep=\parskip,leftmargin=*,align=left,topsep=0.1cm}
\newlist{remlist}{enumerate}{2}
\setlist[remlist,1]{label={(\roman*)},ref={(\roman*)},itemsep=\parskip,%
  leftmargin=*,align=left}
\setlist[remlist,2]{label={(\alph*)},ref={(\alph*)},itemsep=\parskip,%
  leftmargin=*,align=left,topsep=0.1cm}

\numberwithin{equation}{section}

\newtheorem{cor}[equation]{Corollary}
\newtheorem{lem}[equation]{Lemma}
\newtheorem{prop}[equation]{Proposition}

\newtheorem{thm}[equation]{Theorem}

\newtheorem*{claim*}{Claim}

\theoremstyle{definition}

\newtheorem{rem}[equation]{Remark}

\renewcommand{\eqref}[1]{(\ref{#1})}

\usepackage{tikz}
\usetikzlibrary{matrix}
\usepackage{tikz-cd}

\usepackage{xspace}

\usepackage{xifthen}

\usepackage{xparse}

\newcommand{\nc}{\newcommand}
\nc{\renc}{\renewcommand}
\nc{\ssec}{\subsection}
\nc{\sssec}{\subsubsection}
\nc{\on}{\operatorname}
\nc{\term}[1]{#1\xspace}

\newcommand{\p}{\mathfrak{p}}

\DeclareMathSymbol{A}{\mathalpha}{operators}{`A}
\DeclareMathSymbol{B}{\mathalpha}{operators}{`B}
\DeclareMathSymbol{C}{\mathalpha}{operators}{`C}
\DeclareMathSymbol{D}{\mathalpha}{operators}{`D}
\DeclareMathSymbol{E}{\mathalpha}{operators}{`E}
\DeclareMathSymbol{F}{\mathalpha}{operators}{`F}
\DeclareMathSymbol{G}{\mathalpha}{operators}{`G}
\DeclareMathSymbol{H}{\mathalpha}{operators}{`H}
\DeclareMathSymbol{I}{\mathalpha}{operators}{`I}
\DeclareMathSymbol{J}{\mathalpha}{operators}{`J}
\DeclareMathSymbol{K}{\mathalpha}{operators}{`K}
\DeclareMathSymbol{L}{\mathalpha}{operators}{`L}
\DeclareMathSymbol{M}{\mathalpha}{operators}{`M}
\DeclareMathSymbol{N}{\mathalpha}{operators}{`N}
\DeclareMathSymbol{O}{\mathalpha}{operators}{`O}
\DeclareMathSymbol{P}{\mathalpha}{operators}{`P}
\DeclareMathSymbol{Q}{\mathalpha}{operators}{`Q}
\DeclareMathSymbol{R}{\mathalpha}{operators}{`R}
\DeclareMathSymbol{S}{\mathalpha}{operators}{`S}
\DeclareMathSymbol{T}{\mathalpha}{operators}{`T}
\DeclareMathSymbol{U}{\mathalpha}{operators}{`U}
\DeclareMathSymbol{V}{\mathalpha}{operators}{`V}
\DeclareMathSymbol{W}{\mathalpha}{operators}{`W}
\DeclareMathSymbol{X}{\mathalpha}{operators}{`X}
\DeclareMathSymbol{Y}{\mathalpha}{operators}{`Y}
\DeclareMathSymbol{Z}{\mathalpha}{operators}{`Z}

\nc{\sA}{\ensuremath{\mathcal{A}}\xspace}
\nc{\sB}{\ensuremath{\mathcal{B}}\xspace}
\nc{\sC}{\ensuremath{\mathcal{C}}\xspace}
\nc{\sD}{\ensuremath{\mathcal{D}}\xspace}
\nc{\sE}{\ensuremath{\mathcal{E}}\xspace}
\nc{\sF}{\ensuremath{\mathcal{F}}\xspace}
\nc{\sG}{\ensuremath{\mathcal{G}}\xspace}
\nc{\sH}{\ensuremath{\mathcal{H}}\xspace}
\nc{\sI}{\ensuremath{\mathcal{I}}\xspace}
\nc{\sJ}{\ensuremath{\mathcal{J}}\xspace}
\nc{\sK}{\ensuremath{\mathcal{K}}\xspace}
\nc{\sL}{\ensuremath{\mathcal{L}}\xspace}
\nc{\sM}{\ensuremath{\mathcal{M}}\xspace}
\nc{\sN}{\ensuremath{\mathcal{N}}\xspace}
\nc{\sO}{\ensuremath{\mathcal{O}}\xspace}
\nc{\sP}{\ensuremath{\mathcal{P}}\xspace}
\nc{\sQ}{\ensuremath{\mathcal{Q}}\xspace}
\nc{\sR}{\ensuremath{\mathcal{R}}\xspace}
\nc{\sS}{\ensuremath{\mathcal{S}}\xspace}
\nc{\sT}{\ensuremath{\mathcal{T}}\xspace}
\nc{\sU}{\ensuremath{\mathcal{U}}\xspace}
\nc{\sV}{\ensuremath{\mathcal{V}}\xspace}
\nc{\sW}{\ensuremath{\mathcal{W}}\xspace}
\nc{\sX}{\ensuremath{\mathcal{X}}\xspace}
\nc{\sY}{\ensuremath{\mathcal{Y}}\xspace}
\nc{\sZ}{\ensuremath{\mathcal{Z}}\xspace}

\nc{\bA}{\ensuremath{\mathbf{A}}\xspace}
\nc{\bB}{\ensuremath{\mathbf{B}}\xspace}
\nc{\bC}{\ensuremath{\mathbf{C}}\xspace}
\nc{\bD}{\ensuremath{\mathbf{D}}\xspace}
\nc{\bE}{\ensuremath{\mathbf{E}}\xspace}
\nc{\bF}{\ensuremath{\mathbf{F}}\xspace}
\nc{\bG}{\ensuremath{\mathbf{G}}\xspace}
\nc{\bH}{\ensuremath{\mathbf{H}}\xspace}
\nc{\bI}{\ensuremath{\mathbf{I}}\xspace}
\nc{\bJ}{\ensuremath{\mathbf{J}}\xspace}
\nc{\bK}{\ensuremath{\mathbf{K}}\xspace}
\nc{\bL}{\ensuremath{\mathbf{L}}\xspace}
\nc{\bM}{\ensuremath{\mathbf{M}}\xspace}
\nc{\bN}{\ensuremath{\mathbf{N}}\xspace}
\nc{\bO}{\ensuremath{\mathbf{O}}\xspace}
\nc{\bP}{\ensuremath{\mathbf{P}}\xspace}
\nc{\bQ}{\ensuremath{\mathbf{Q}}\xspace}
\nc{\bR}{\ensuremath{\mathbf{R}}\xspace}
\nc{\bS}{\ensuremath{\mathbf{S}}\xspace}
\nc{\bT}{\ensuremath{\mathbf{T}}\xspace}
\nc{\bU}{\ensuremath{\mathbf{U}}\xspace}
\nc{\bV}{\ensuremath{\mathbf{V}}\xspace}
\nc{\bW}{\ensuremath{\mathbf{W}}\xspace}
\nc{\bX}{\ensuremath{\mathbf{X}}\xspace}
\nc{\bY}{\ensuremath{\mathbf{Y}}\xspace}
\nc{\bZ}{\ensuremath{\mathbf{Z}}\xspace}

\nc{\dA}{\ensuremath{\mathds{A}}\xspace}
\nc{\dB}{\ensuremath{\mathds{B}}\xspace}
\nc{\dC}{\ensuremath{\mathds{C}}\xspace}
\nc{\dD}{\ensuremath{\mathds{D}}\xspace}
\nc{\dE}{\ensuremath{\mathds{E}}\xspace}
\nc{\dF}{\ensuremath{\mathds{F}}\xspace}
\nc{\dG}{\ensuremath{\mathds{G}}\xspace}
\nc{\dH}{\ensuremath{\mathds{H}}\xspace}
\nc{\dI}{\ensuremath{\mathds{I}}\xspace}
\nc{\dJ}{\ensuremath{\mathds{J}}\xspace}
\nc{\dK}{\ensuremath{\mathds{K}}\xspace}
\nc{\dL}{\ensuremath{\mathds{L}}\xspace}
\nc{\dM}{\ensuremath{\mathds{M}}\xspace}
\nc{\dN}{\ensuremath{\mathds{N}}\xspace}
\nc{\dO}{\ensuremath{\mathds{O}}\xspace}
\nc{\dP}{\ensuremath{\mathds{P}}\xspace}
\nc{\dQ}{\ensuremath{\mathds{Q}}\xspace}
\nc{\dR}{\ensuremath{\mathds{R}}\xspace}
\nc{\dS}{\ensuremath{\mathds{S}}\xspace}
\nc{\dT}{\ensuremath{\mathds{T}}\xspace}
\nc{\dU}{\ensuremath{\mathds{U}}\xspace}
\nc{\dV}{\ensuremath{\mathds{V}}\xspace}
\nc{\dW}{\ensuremath{\mathds{W}}\xspace}
\nc{\dX}{\ensuremath{\mathds{X}}\xspace}
\nc{\dY}{\ensuremath{\mathds{Y}}\xspace}
\nc{\dZ}{\ensuremath{\mathds{Z}}\xspace}

\nc{\bbA}{\ensuremath{\mathbb{A}}\xspace}
\nc{\bbB}{\ensuremath{\mathbb{B}}\xspace}
\nc{\bbC}{\ensuremath{\mathbb{C}}\xspace}
\nc{\bbD}{\ensuremath{\mathbb{D}}\xspace}
\nc{\bbE}{\ensuremath{\mathbb{E}}\xspace}
\nc{\bbF}{\ensuremath{\mathbb{F}}\xspace}
\nc{\bbG}{\ensuremath{\mathbb{G}}\xspace}
\nc{\bbH}{\ensuremath{\mathbb{H}}\xspace}
\nc{\bbI}{\ensuremath{\mathbb{I}}\xspace}
\nc{\bbJ}{\ensuremath{\mathbb{J}}\xspace}
\nc{\bbK}{\ensuremath{\mathbb{K}}\xspace}
\nc{\bbL}{\ensuremath{\mathbb{L}}\xspace}
\nc{\bbM}{\ensuremath{\mathbb{M}}\xspace}
\nc{\bbN}{\ensuremath{\mathbb{N}}\xspace}
\nc{\bbO}{\ensuremath{\mathbb{O}}\xspace}
\nc{\bbP}{\ensuremath{\mathbb{P}}\xspace}
\nc{\bbQ}{\ensuremath{\mathbb{Q}}\xspace}
\nc{\bbR}{\ensuremath{\mathbb{R}}\xspace}
\nc{\bbS}{\ensuremath{\mathbb{S}}\xspace}
\nc{\bbT}{\ensuremath{\mathbb{T}}\xspace}
\nc{\bbU}{\ensuremath{\mathbb{U}}\xspace}
\nc{\bbV}{\ensuremath{\mathbb{V}}\xspace}
\nc{\bbW}{\ensuremath{\mathbb{W}}\xspace}
\nc{\bbX}{\ensuremath{\mathbb{X}}\xspace}
\nc{\bbY}{\ensuremath{\mathbb{Y}}\xspace}
\nc{\bbZ}{\ensuremath{\mathbb{Z}}\xspace}

\nc{\mrm}[1]{\ensuremath{\mathrm{#1}}\xspace}
\nc{\mbf}[1]{\ensuremath{\mathbf{#1}}\xspace}
\nc{\mcal}[1]{\ensuremath{\mathcal{#1}}\xspace}
\nc{\msc}[1]{\ensuremath{\mathscr{#1}}\xspace}

\renc{\bar}[1]{\overline{#1}}

\let\sectsign\S
\let\S\relax

\nc{\sub}{\subset}
\nc{\too}{\longrightarrow}
\nc{\hook}{\hookrightarrow}
\nc*{\hooklongrightarrow}{\ensuremath{\lhook\joinrel\relbar\joinrel\rightarrow}}
\nc{\hooklong}{\hooklongrightarrow}
\nc{\twoheadlongrightarrow}{\relbar\joinrel\twoheadrightarrow}
\nc{\shiso}{\approx}
\nc{\isoto}{\xrightarrow{\sim}}
\nc{\isofrom}{\xleftarrow{\sim}}
\renc{\ge}{\geqslant}
\renc{\le}{\leqslant}
\renc{\geq}{\geqslant}
\renc{\leq}{\leqslant}

\nc{\id}{\mathrm{id}}

\DeclareMathOperator{\Hom}{\mathrm{Hom}}
\nc{\uHom}{{\smash{\Hom}}}
\DeclareMathOperator{\Maps}{\mathrm{Maps}}

\DeclareMathOperator{\End}{\mathrm{End}}

\nc{\Pre}{\mathrm{PSh}{}}
\nc{\Shv}{\mathrm{Shv}{}}
\nc{\uEnd}{{\smash{\End}}}

\renc{\lim}{\operatorname*{lim}}
\nc{\colim}{\operatorname*{colim}}
\nc{\Cofib}{\on{Cofib}}
\nc{\Fib}{\on{Fib}}
\nc{\initial}{\varnothing}
\nc{\op}{\mathrm{op}}

\renc{\coprod}{\sqcup}

\nc{\bDelta}{\mbf{\Delta}}
\nc{\DM}{\mbf{DM}}
\nc{\eff}{\mathrm{eff}}
\nc{\veff}{\mathrm{veff}}
\nc{\cyc}{{\mrm{cyc}}}
\nc{\corr}{{\on{corr}}}
\nc{\ft}{\mrm{ft}}
\nc{\flf}{\mrm{flf}}
\nc{\fet}{{\mrm{f\acute et}}}
\nc{\fsyn}{{\mrm{fsyn}}}
\nc{\syn}{{\mrm{syn}}}
\nc{\lci}{{\mrm{lci}}}
\nc{\Perf}{\mbf{Perf}}
\nc{\perf}{\mrm{perf}}
\nc{\oblv}{\mrm{oblv}}
\nc{\exact}{\on{exact}}

\nc{\F}{{\on{F}}}
\nc{\clopen}{{\mrm{clopen}}}
\nc{\B}{\mrm{B}}
\nc{\D}{\mrm{D}}
\nc{\Fin}{\on{Fin}}
\nc{\fin}{\mrm{fin}}
\nc{\Cut}{\on{Cut}}
\nc{\Cart}{\on{Cart}}
\nc{\pairs}{\mathsf{pairs}}
\nc{\Pairs}{\mathrm{Pair}}
\nc{\Trip}{\mathrm{Trip}}
\nc{\Lab}{\mathrm{Lab}}
\nc{\SL}{\mathrm{SL}}
\nc{\coCart}{\mathrm{coCart}}
\nc{\RKE}{\mathrm{RKE}}
\nc{\strict}{\mathrm{strict}}
\nc{\Emb}{\mathrm{Emb}}
\nc{\Split}{\mathrm{Split}}
\nc{\Set}{\mathrm{Set}}
\nc{\sSets}{\mathrm{sSets}}
\nc{\pb}{\mathrm{pb}}
\nc{\fib}{\mathrm{fib}}
\nc{\diff}{\mrm{diff}}
\nc{\gp}{\mrm{gp}}
\nc{\map}{\mrm{map}}

\nc{\mgp}{\mrm{mot-gp}}
\nc{\FSyn}{\mrm{FSyn}}
\nc{\FEt}{\mrm{FEt}}
\nc{\Spc}{\mrm{Spc}}
\nc{\Ob}{\mrm{Ob}}
\nc{\Spt}{\mrm{Spt}}
\nc{\T}{\bT}
\nc{\suspinf}{\Sigma^\infty}
\nc{\h}{\mrm{h}}
\nc{\uhom}{{\mathrm{Hom}}}
\nc{\umap}{{\mathrm{Maps}}}
\renc{\H}{\bH}
\nc{\Einfty}{{\sE_\infty}}
\nc{\Eone}{{\sE_1}}
\nc{\Stab}{\mrm{Stab}}
\nc{\lax}{{\mrm{lax}}}
\nc{\cocart}{{\mrm{cocart}}}
\nc{\Sch}{\mrm{Sch}}
\nc{\Fr}{\on{Fr}}
\nc{\A}{\mathbf{A}}
\nc{\N}{\mathbf{N}}
\nc{\Z}{\mathbf{Z}}
\nc{\Q}{\mathbf{Q}}
\nc{\Oo}{\mathcal{O}} 
\nc{\Fscr}{\mathcal{F}}
\nc{\Gscr}{\mathcal{G}}
\nc{\Ll}{\mathcal{L}} 
\nc{\Mm}{\mathcal{M}} 
\nc{\mm}{\mathrm{m}} 
\nc{\K}{\mrm{K}} 
\nc{\W}{\mrm{W}} 
\nc{\red}{{\on{red}}}
\nc{\Voev}{{\on{Voev}}}
\nc{\Corr}{\mrm{Corr}}
\nc{\Span}{\mathbf{Corr}}
\nc{\Gap}{\mrm{Gap}}
\nc{\Corrfr}{\Corr^{\fr}}
\nc{\Corrvfr}{\Corr^{\Vfr}}
\nc{\Spec}{\on{Spec}}
\nc{\Sm}{\on{Sm}}
\nc{\Gm}{\mathbf{G}_{\on{m}}}
\renc{\P}{\bP}
\nc{\nis}{\mathrm{nis}}
\nc{\zar}{\mathrm{zar}}
\nc{\et}{\mathrm{\acute et}}
\nc{\all}{\mathrm{all}}
\nc{\fold}{\mathrm{fold}}
\nc{\Fun}{\mathrm{Fun}}
\nc{\Ho}{\mathrm{Ho}}
\nc{\Segal}{\mathrm{Segal}}
\nc{\Mon}{\mrm{Mon}{}}
\nc{\Ab}{\mrm{Ab}}
\nc{\Sh}{\on{Sh}}
\nc{\M}{\mrm{M}}
\nc{\Lhtp}{L_{\A^1}}
\nc{\Lmot}{L_{\mrm{mot}}}
\nc{\mot}{\mrm{mot}}
\nc{\SH}{\mbf{SH}}
\nc{\RR}{\mbf{R}}
\nc{\CC}{\mbf{C}}
\nc{\Mod}{\mbf{Mod}}
\nc{\QCoh}{\mbf{QCoh}}
\nc{\MonUnit}{\mbf{1}}
\nc{\1}{\mbf{1}}
\nc{\tr}{\on{tr}}
\nc{\cotr}{\mrm{cotr}}
\nc{\vop}{\mrm{vop}}
\nc{\fr}{{\on{fr}}}
\nc{\Ar}{\mrm{Ar}}
\nc{\Vfr}{\on{Vfr}}
\nc{\frdiff}{{\on{frdiff}}}
\nc{\frGys}{\on{frGys}}
\nc{\SHfr}{\SH^{\fr}}
\nc{\SHfrdiff}{\SH^{\frdiff}}
\nc{\SHfrGys}{\SH^{\frGys}}
\nc{\InftyCat}{(\mathrm{\infty,1)\textnormal{-}Cat}}
\nc{\TriCat}{\mathrm{TriCat}}
\nc{\oneCat}{\mathrm{1\textnormal{-}Cat}}
\nc{\Cat}{\mathrm{Cat}}
\nc{\Th}{\on{Th}}

\nc{\CMon}{\mrm{CMon}{}}
\nc{\CAlg}{\mrm{CAlg}{}}
\nc{\MGL}{\mrm{MGL}}
\nc{\Seg}{\mrm{Seg}{}}
\nc{\GW}{\mrm{GW}{}}
\nc{\Tw}{\mrm{Tw}}
\nc{\sslash}{/\mkern-6mu/}
\nc{\PrL}{\mrm{Pr}^\mrm{L}}
\nc{\PrR}{\mrm{Pr}^\mrm{R}}
\nc{\pr}{\mrm{pr}}
\let\phi\varphi

\nc\efr{\mrm{efr}}
\nc\nfr{\mrm{nfr}}
\nc\dfr{\mrm{fr}}
\nc\tfr{\mrm{tfr}}
\nc\Vect{\mrm{Vect}}
\nc\sVect{\mrm{sVect}}
\nc{\fix}{\mrm{fix}}
\nc{\ho}{\mrm{h}}
\nc\Mfd{\mrm{Mfd}}
\nc{\PSh}{\mrm{PSh}}
\nc{\hzmw}{H \tilde\Z{}}
\nc{\Cor}{\mrm{Cor}{}}
\nc{\cormw}{\mrm{\widetilde{Cor}}{}}
\nc{\Chw}{\mrm{\widetilde{CH}}{}}
\nc{\Ex}{\mrm{Ex}}
\nc{\BM}{\mrm{BM}}

\nc{\Pic}{\mrm{Pic}}
\nc{\pur}{\mathfrak p}
\nc{\angles}[1]{\langle #1\rangle}
\nc{\inv}[1]{[\tfrac{1}{#1}]}
\nc{\pinv}{\inv{p}}
\nc{\cinv}{\inv{p}}
\nc{\Sph}{\on{Sph}}
\nc{\KGL}{\mrm{KGL}}
\nc{\KH}{\mrm{KH}}
\nc{\Flag}{\mrm{Flag}}
\nc{\Pro}{\mrm{Pro}}
\nc{\Frac}{\mrm{Frac}}
\newcommand{\fp}{\mathrm{fp}}

\nc{\arc}{\mrm{arc}}
\nc{\rarc}{\mrm{rarc}}
\nc{\cdarc}{\mrm{cdarc}}
\nc{\vv}{\mrm{v}}
\nc{\rv}{\mrm{rv}}
\nc{\cdv}{\mrm{cdv}}
\nc{\hh}{\mrm{h}}
\nc{\cdh}{\mrm{cdh}}
\nc{\rh}{\mathrm{rh}}
\nc{\Et}{\mathrm{Et}}
\nc{\Nis}{\mathrm{Nis}}
\nc{\Zar}{\mathrm{Zar}}
\nc{\cdp}{\mathrm{cdp}}

\nc{\RZ}{\mathrm{RZ}}
\nc{\qcqs}{\mathrm{qcqs}}
\nc{\aff}{\mathrm{aff}}
\nc{\cl}{\mathrm{cl}}
\nc{\Val}{\mathrm{Val}}
\nc{\GFin}{\mathrm{GFin}{}}
\nc{\Proj}{\mathrm{Proj}}
\nc{\Ind}{\mrm{Ind}}
\let\emptyset=\varnothing
\nc{\bigH}{\H{}}
\nc{\bigSH}{\SH{}}

\nc{\DA}{\mathbf{DA}}
\nc{\ct}{\mathrm{ct}}
\nc{\lc}{\mathrm{lc}}
\nc{\cd}{\mathrm{cd}}
\nc{\lcd}{\mathrm{lcd}}
\nc{\vcd}{\mathrm{vcd}}

\nc{\inftyCat}{\term{$\infty$-category}}
\nc{\inftyCats}{\term{$\infty$-categories}}

\nc{\inftyOneCat}{\term{$(\infty,1)$-category}}
\nc{\inftyOneCats}{\term{$(\infty,1)$-categories}}

\nc{\inftyGrpd}{\term{$\infty$-groupoid}}
\nc{\inftyGrpds}{\term{$\infty$-groupoids}}

\nc{\inftyTop}{\term{$\infty$-topos}}
\nc{\inftyTops}{\term{$\infty$-toposes}}

\nc{\inftyTwoCat}{\term{$(\infty,2)$-category}}
\nc{\inftyTwoCats}{\term{$(\infty,2)$-categories}}

\title{Milnor excision for motivic spectra}

\author[E. Elmanto]{Elden Elmanto}
\address{Department of Mathematics\\
Harvard University\\
1 Oxford St.\\
Cambridge, MA 02138\\
USA}
\email{\href{mailto:elmanto@math.harvard.edu}{elmanto@math.harvard.edu}}
\urladdr{\url{https://www.eldenelmanto.com/}}

\author[M. Hoyois]{Marc Hoyois}
\address{Fakultät für Mathematik\\
Universität Regensburg\\
Universitätsstr. 31\\
93040 Regensburg\\
Germany}
\email{\href{mailto:marc.hoyois@ur.de}{marc.hoyois@ur.de}}
\urladdr{\url{http://www.mathematik.ur.de/hoyois/}}

\author[R. Iwasa]{Ryomei Iwasa}
\address{K\o benhavns Universitet\\
Institut for Matematiske Fag\\
Universitetsparken 5\\
2100 K\o benhavn\\
Denmark}
\email{\href{mailto:ryomei@math.ku.dk}{ryomei@math.ku.dk}}
\urladdr{http://ryomei.com/}

\author[S. Kelly]{Shane Kelly}
\address{Department of Mathematics\\
Tokyo Institute of Technology\\
2-12-1 Ookayama, Meguro-ku\\
Tokyo 152-8551, Japan}
\email{\href{mailto:shanekelly@math.titech.ac.jp}{shanekelly@math.titech.ac.jp}}
\urladdr{http://www.math.titech.ac.jp/~shanekelly/}

\date{\today}

\begin{document}

\begin{abstract} 
	We prove that the $\infty$-category of motivic spectra satisfies Milnor excision: if $A\to B$ is a morphism of commutative rings sending an ideal $I\subset A$ isomorphically onto an ideal of $B$, then a motivic spectrum over $A$ is equivalent to a pair of motivic spectra over $B$ and $A/I$ that are identified over $B/IB$. Consequently, any cohomology theory represented by a motivic spectrum satisfies Milnor excision. We also prove Milnor excision for Ayoub's étale motives over schemes of finite virtual cohomological dimension.
\end{abstract}

\maketitle

\parskip 0pt
\tableofcontents

\parskip 0.2cm

\section{Introduction}
For $S$ a scheme, let $\SH(S)$ be the $\infty$-category of motivic spectra over $S$.

\begin{thm}\label{thm:main}
	The presheaf of $\infty$-categories $\SH(-)\colon \Sch^\op \to\Cat_\infty$ satisfies Milnor excision.
\end{thm}

This means the following \cite[Definition 3.2.3]{cdarc}: given a cartesian square of schemes
\[
\begin{tikzcd}
	W \ar{d}[swap]{g} \ar{r}{k} & Y \ar{d}{f} \\
	Z \ar{r}{i} & X
\end{tikzcd}
\]
where $f$ is affine, $i$ is a closed immersion, and the induced map $Y\sqcup_WZ \to X$ is an isomorphism, the square of $\infty$-categories
\[
\begin{tikzcd}
	\SH(X) \ar{d}[swap]{f^*} \ar{r}{i^*} & \SH(Z) \ar{d}{g^*} \\
	\SH(Y) \ar{r}{k^*} & \SH(W)
\end{tikzcd}
\]
is cartesian.

If $E\in\SH(S)$ is a motivic spectrum and $X$ is an $S$-scheme, we denote by $E(X)\in\Spt$ the mapping spectrum from $\1_X$ to $E_X$ in $\SH(X)$. An immediate consequence of Theorem~\ref{thm:main} is the following:

\begin{cor}\label{cor:main}
	Let $S$ be a scheme and $E\in\SH(S)$. Then the presheaf of spectra $E(-)\colon \Sch_S^\op\to \Spt$ satisfies Milnor excision.
\end{cor}

Setting $E^{p,q}(X)= [\1_X, \Sigma^{p,q}E_X]$, Corollary~\ref{cor:main} yields a long exact sequence of abelian groups
\[
\cdots \rightarrow E^{p,q}(X) \rightarrow E^{p,q}(Z) \oplus E^{p,q}(Y) \rightarrow E^{p,q}(W)\rightarrow E^{p+1,q}(X) \rightarrow \cdots.
\]
Milnor excision was originally considered in the context of algebraic K-theory by Milnor, who established such a long exact sequence for the algebraic K-groups in degrees $\leq 1$ \cite[Theorem 3.3]{milnor-ktheory}.
This gave a simple proof of Rim's theorem asserting an isomorphism between $K_0$ of the group ring $\Z[C_p]$ and of the ring of cyclotomic integers $\Z[e^{2\pi i/p}]$. While excision does not hold in general for $K$-theory \cite{swan}, Suslin–Wodzicki \cite{suslin-wodzicki} and Suslin \cite{suslin-excision} isolated conditions on squares of (possibly noncommutative) rings for which excision holds. Switching perspective by considering excision for other cohomology theories in algebraic geometry has also been fruitful, notably in Swan's study of the relationship between seminormality and the $\A^1$-invariance of the Picard group \cite{swan-seminormality}. In recent years, Milnor excision and its variants have had influential roles in settling old conjectures (such as Weibel's conjecture \cite{KST}) and bringing new insights into cohomology theories in algebraic geometry and other subjects \cite{arc,land-tamme}.

Corollary~\ref{cor:main} furnishes a large collection of cohomology theories for which excision holds, vastly generalizing our previous work for motivic cohomology \cite[Theorem D]{cdarc}. 
It also recovers Weibel's excision theorem for the homotopy K-theory of commutative rings \cite[Theorem 2.1]{Weibel}, as well as Bhatt and Mathew's excision theorem for étale cohomology with coefficients in a torsion sheaf \cite[Theorem 5.4]{arc}, provided that the torsion is coprime to the residual characteristics.
One can regard Corollary~\ref{cor:main} as a commutative analogue of Land and Tamme's excision theorem for noncommutative motives \cite[Theorem B]{land-tamme}.
If $S=\Spec R$, an equivalent formulation of Corollary~\ref{cor:main} is that the canonical extension of $E(-)$ to nonunital commutative $R$-algebras sends short exact sequences to fiber sequences (cf.\ \cite[Remark 3.2.6]{cdarc}). 

Combining Corollary~\ref{cor:main} with \cite[Lemma 3.5(ii)]{kelly-morrow}, we obtain the following result, which verifies the property (G2) from \cite[Theorem 1]{shane-better} for the homotopy presheaves of any motivic spectrum over a perfect field:

\begin{cor}
	Let $k$ be a perfect field and $E\in\SH(k)$. For every valuation ring $V$ over $k$, henselian along an ideal $I\subset V$, the map $\pi_* E(V) \to \pi_* E(V/I)$ is surjective.
\end{cor}

Here is a brief outline of the proof of Theorem~\ref{thm:main}. Using the main result of \cite{cdarc}, we first reduce it to the statement that $\SH(-)$ satisfies v-excision, which is special case of Milnor excision involving valuation rings (Theorem~\ref{thm:v-excision}). In turn, this is equivalent to a certain unexpected functorial property of $\SH(-)$ with respect to localizations of valuation rings (Equation~\eqref{eqn:SH-transformation}). 
To prove the latter, the main idea is to pass to the larger $\infty$-category $\bigSH_\cdh(-)$ built from the cdh site instead of the smooth Nisnevich site; the cdh descent property of motivic spectra proved by Cisinski \cite{Cisinski} implies that $\bigSH_\cdh(-)$ contains $\SH(-)$ as a full subcategory. This allows us to take advantage of the fact that pushforward along an open immersion preserves cdh-local equivalences (Lemma~\ref{lem:cdh-cocontinuous}). This fact further reduces the question to the level of presheaves (Lemma~\ref{lem:key}), where it boils down to a simple geometric property of valuation rings (Lemma~\ref{lem:interval}).

\section{Reduction of Milnor excision to v-excision} 

Since $\SH(-)$ is a Zariski sheaf, we need only prove Theorem~\ref{thm:main} for qcqs schemes.
We shall do this by applying \cite[Theorem 3.3.4]{cdarc} to the presheaf
\[
\SH(-)^\omega\colon\Sch^\mathrm{qcqs,op} \to \Cat_\infty,
\]
where $\SH(X)^\omega\subset\SH(X)$ is the full subcategory of compact objects. To explain how, we need a pair of technical lemmas.

\begin{lem}\label{lem:cg-pullback}
	Let $Q$ be a commutative square of small $\infty$-categories:
	\[
	\begin{tikzcd}
		\sA \ar{d}[swap]{h} \ar{r}{f} & \sB \ar{d}{k} \\
		\sC \ar{r}{g} & \sD\rlap.
	\end{tikzcd}
	\]
	\begin{enumerate}
		\item If $\sA$ is idempotent complete and $\Ind(Q)$ is cartesian, then $Q$ is cartesian.
		\item Suppose that $Q$ is a square of stable $\infty$-categories and exact functors. If $g$ has a fully faithful right adjoint and $Q$ is cartesian, then $\Ind(Q)$ is cartesian.
	\end{enumerate}
\end{lem}

\begin{proof}
	(i) This follows from \cite[Lemma 5.4.5.7(2)]{HTT}. 
	
	(ii) Form the cartesian square
	\[
	\begin{tikzcd}
		\sE \ar{d}[swap]{h} \ar{r}{f} & \Ind(\sB) \ar{d}{k} \\
		\Ind(\sC) \ar{r}{g} & \Ind(\sD)\rlap.
	\end{tikzcd}
	\]
	It follows from \cite[Lemma 5.4.5.7(2)]{HTT} that $\sA\subset \sE^\omega$, so it suffices to show that $\sA$ generates $\sE$. 
	Let $e\in\sE$ be such that $\Maps(a,e)=0$ for all $a\in\sA$; we must show that $e=0$.
	Let $r$ be the right adjoint of $g$.
	If $a\in\sA$ is in the kernel of $f$ (equivalently, of $g$), $\Maps(a,e)\simeq \Maps(a,h(e))$. Hence, $h(e)$ is right orthogonal to the kernel of $g$, so $h(e)=rgh(e)$. 
	On the other hand, if $a$ is the image of $b\in\sB$ by the functor $\sB\to\sA$ induced by $r\circ k$, then 
	\[
	\Maps(a,e) \simeq \Maps(b,f(e)) \times_{\Maps(k(b),kf(e))} \Maps(rk(b),h(e)) \simeq \Maps(b,f(e)),
	\]
	since $h(e)=rkf(e)$ and $r$ is fully faithful. This shows that $f(e)=0$, hence also $h(e)=0$, hence $e=0$.
\end{proof}

\begin{lem}\label{lem:cg-filtered}
	Let $\sK$ be a filtered $\infty$-category and $D\colon \sK^\triangleright\to \Cat_\infty$ a diagram of small $\infty$-categories with finite colimits and right exact functors. Let $\widehat D\colon (\sK^\op)^\triangleleft \to\Cat_\infty$ be the diagram obtained from $D$ by applying $\Ind$ and passing to right adjoints.
	\begin{enumerate}
		\item If $D(k)$ is idempotent complete for all $k\in\sK$ and $\widehat D$ is a limit diagram, then $D$ is a colimit diagram.
		\item If $D$ is a colimit diagram, then $\widehat D$ is a limit diagram.
	\end{enumerate}
\end{lem}

\begin{proof}
	By \cite[Proposition 5.5.7.11]{HTT}, $D$ is a colimit diagram in $\Cat_\infty$ if and only if it is so in $\Cat_\infty^\mathrm{rex}$.
	By \cite[Proposition 5.5.7.6]{HTT}, $\widehat D$ is a limit diagram in $\Cat_\infty$ if and only if it is so in $\Pr^R_\omega$. Passing to adjoints gives an equivalence $(\Pr^R_\omega)^\op \simeq \Pr^L_\omega$ \cite[Notation 5.5.7.7]{HTT}. 
	Assertion (i) now follows from \cite[Proposition 5.5.7.8]{HTT} and \cite[Lemma 7.3.5.10]{HA}, while assertion (ii) follows from \cite[Proposition 5.5.7.10]{HTT}.
\end{proof}

Recall that $\SH(-)$ is a cdh sheaf \cite[Proposition 6.24]{hoyois-sixops} and that $\SH(X)$ is compactly generated when $X$ is qcqs \cite[Proposition C.12(1,2)]{HoyoisGLV}. Since cdh descent on $\Sch^\mathrm{qcqs}$ is equivalent to certain squares being taken to cartesian squares \cite[Proposition 2.1.5(2)]{cdarc}, it follows from Lemma~\ref{lem:cg-pullback}(i) that $\SH(-)^\omega$ is a cdh sheaf on $\Sch^\mathrm{qcqs}$. The reason for passing to compact objects is that $\SH(-)^\omega$ is also a finitary presheaf, i.e., it transforms limits of cofiltered diagrams of qcqs schemes with affine transition maps into colimits of $\infty$-categories: this follows from \cite[Proposition C.12(4)]{HoyoisGLV} and Lemma~\ref{lem:cg-filtered}(i).
Since the $\infty$-category of small $\infty$-categories is compactly generated, the presheaf $\SH(-)^\omega$ (more precisely, its right Kan extension from $\Sch^\mathrm{qcqs}$ to $\Sch$) satisfies the assumptions of \cite[Theorem 3.3.4]{cdarc} over the base $\Spec \Z$. The conclusion is that $\SH(-)^\omega$ satisfies Milnor excision if and only if it satisfies henselian v-excision. 
If $i\colon Z\hook X$ is a closed immersion of qcqs schemes, the functor $i^*\colon \SH(X)^\omega\to \SH(Z)^\omega$ has a fully faithful right adjoint, since $i_*$ preserves compact objects (by localization \cite[Proposition C.10]{HoyoisGLV}).
Hence, by Lemma~\ref{lem:cg-pullback}, $\SH(-)^\omega$ satisfies Milnor excision or henselian v-excision if and only if $\SH(-)$ does. Theorem~\ref{thm:main} is therefore reduced to the following theorem:

\begin{thm}\label{thm:v-excision}
	Let $V$ be a valuation ring and $\p\subset V$ a prime ideal. Then the following square of $\infty$-categories is cartesian:
	\[
	\begin{tikzcd}
		\SH(V) \ar{r} \ar{d} & \SH(V_\p) \ar{d} \\ \SH(V/\p) \ar{r} & \SH(\kappa(\p))\rlap.
	\end{tikzcd}
	\]
\end{thm}

Moreover, $\SH(-)^\omega$ being finitary, it is enough to prove Theorem~\ref{thm:v-excision} for $V$ a valuation ring of finite rank \cite[Remark 3.3.3]{cdarc}. In this case, $\Spec V_\p\to \Spec V$ is an open immersion.

\section{v-excision in terms of the six operations}
\label{sec:six-op}

In this section, we reduce the finite rank case of Theorem~\ref{thm:v-excision} to a statement involving the formalism of six operations for motivic spectra. In fact, this reduction only uses that $\SH(-)$ is a presheaf of presentable $\infty$-categories satisfying the localization property, meaning that for $i\colon Z\hook X$ a closed immersion with open complement $j\colon U\hook X$, the functors $i_*$ and $j_*$ are fully faithful and the image of $i_*$ coincides with the kernel of $j^*$ \cite[Proposition C.10]{HoyoisGLV}.

We begin by examining under which conditions $\SH(-)$ sends a square to a cartesian square.
A commutative square of schemes
\[
\begin{tikzcd}
	W \ar{r}{k} \ar{d}[swap]{g} \ar{dr}{l} & Y \ar{d}{f} \\ Z \ar{r}{h} & X
\end{tikzcd}
\]
induces an adjunction
\begin{equation}\label{eqn:adjunction}
\SH(X) \rightleftarrows \SH(Y)\times_{\SH(W)} \SH(Z).
\end{equation}
This adjunction is an equivalence if and only if the left adjoint is conservative and the right adjoint is fully faithful, in other words if and only if:
\begin{enumerate}
	\item the functor $(f^*,h^*)\colon \SH(X)\to \SH(Y)\times \SH(Z)$ is conservative;
	\item given $E_{Y}\in \SH(Y)$, $E_{Z}\in \SH(Z)$, $E_W\in\SH(W)$, $k^*E_Y \simeq E_W$, and $g^*E_Z\simeq E_W$, if $E=f_*(E_Y) \times_{l_* E_W} h_*(E_Z)$, then the canonical maps
	\[
	f^*(E) \to E_Y \quad\text{and}\quad h^*(E) \to E_Z
	\]
	are equivalences.
\end{enumerate}
If $f$ is an immersion, then $f_*$ is fully faithful and hence $f^*(E) \simeq E_Y \times_{k_*g^* E_Z} f^*h_*E_Z$. It follows that $f^*(E)\to E_Y$ is an equivalence if and only if the exchange morphism $f^* h_*(E_Z) \to k_* g^*(E_Z)$ is an equivalence.
Thus, if $f$, $g$, $h$, and $k$ are all immersions (so that $k^*$ and $g^*$ are essentially surjective), then (ii) holds if and only if the following exchange transformations are equivalences:
\begin{gather*}
f^* h_* \to k_* g^* \colon \SH(Z) \to \SH(Y), \\
h^* f_* \to g_* k^* \colon \SH(Y) \to \SH(Z).
\end{gather*}

\begin{rem}
	In the adjunction~\eqref{eqn:adjunction}, the left adjoint functor is fully faithful if and only if, for all $E\in\SH(X)$, $E(-)$ converts every smooth base change of the given square to a cartesian square.
	For Milnor squares (which are preserved by smooth base change \cite[Lemma~3.2.9]{cdarc}), this is precisely the content of Corollary~\ref{cor:main}. 
	For abstract blowup squares, this was first proved by Cisinski in \cite[Proposition~3.7]{Cisinski}. The stronger statement that $\SH(-)$ itself sends abstract blowup squares to cartesian squares was proved in \cite[Proposition~6.24]{hoyois-sixops} by verifying conditions (i) and (ii) above.
\end{rem}

Now let $V$ be a valuation ring of finite rank and $\p\subset V$ a prime ideal. Set $X=\Spec V$, $U=\Spec V_\p$, $Z=\Spec V/\p$, and $T=X-Z$. We have a commutative diagram
\[
\begin{tikzcd}
	& U\cap Z \ar{r}{v} \ar{d}{k} & Z \ar{d}{i} \\
	T \ar{r}{t} & U \ar{r}{u} & X
\end{tikzcd}
\]
where the horizontal maps are open immersions and the vertical maps are closed immersions. Since $U$ and $Z$ cover $X$, the functor $\SH(X)\to \SH(U)\times \SH(Z)$ is conservative (by localization). Specializing the above discussion to this situation, we see that Theorem~\ref{thm:v-excision} holds if and only if the base change transformation
\[
i^*u_* \to v_* k^*\colon \SH(U) \to \SH(Z)
\]
is an equivalence (the other transformation $u^*i_*\to k_* v^*$ being an equivalence by the localization sequences for $i$ and $k$). This transformation induces the rightmost morphism in the diagram of localization sequences
\[
\begin{tikzcd}
	\llap{$u_!t_!t^*\simeq{}$}(ut)_!(ut)^* u_* \ar{d} \ar{r} & u_* \ar[equal]{d} \ar{r} & i_*i^*u_* \ar{d} \\
	u_*t_!t^* \ar{r} & u_* \ar{r} & u_*k_*k^*\rlap{${}\simeq i_* v_* k^*$.}
\end{tikzcd}
\]
Since $i_*$ is fully faithful and $t^*$ is surjective, we deduce that Theorem~\ref{thm:v-excision} holds for $\p\subset V$ if and only if the canonical transformation
\begin{equation}\label{eqn:SH-transformation}
u_!t_! \to u_* t_!\colon \SH(T) \to \SH(X)
\end{equation}
is an equivalence.

\section{Proof of v-excision via cdh descent}

We now prove that the transformation~\eqref{eqn:SH-transformation} is an equivalence and thereby the main theorem of the paper. As explained in the introduction, the key insight is to work with the cdh topology. 

Let $\bigH_\cdh(X)$ and $\bigSH_\cdh(X)$ be the analogues of $\H(X)$ and $\SH(X)$ constructed using the cdh site $\Sch_X^\mathrm{lfp}$ instead of the Nisnevich site $\Sm_X$. The inclusion $\Sm_X\subset \Sch_X^\mathrm{lfp}$ induces left adjoint functors $\H(X)\to \bigH_\cdh(X)$ and $\SH(X) \to\bigSH_\cdh(X)$, and the fact that $\SH(-)$ satisfies cdh descent implies that the latter is fully faithful \cite{shcdh}. For $f\colon Y\to X$ any morphism, we have commutative squares
\[
\begin{tikzcd}
	\SH(X) \ar{r}{f^*} \ar[hook]{d} & \SH(Y) \ar[hook]{d} \\
	\bigSH_\cdh(X) \ar{r}{f^*} & \bigSH_\cdh(Y)
\end{tikzcd}
\quad
\begin{tikzcd}
	\SH(X) & \SH(Y) \ar{l}[swap]{f_*} \\
	\bigSH_\cdh(X) \ar{u} & \bigSH_\cdh(Y)\rlap. \ar{u} \ar{l}[swap]{f_*}
\end{tikzcd}
\]
If $f\colon Y\to X$ is smooth, we moreover have a commutative square
\[
\begin{tikzcd}
	\SH(Y) \ar{r}{f_\sharp} \ar[hook]{d} & \SH(X) \ar[hook]{d} \\
	\bigSH_\cdh(Y) \ar{r}{f_\sharp} & \bigSH_\cdh(X)\rlap.
\end{tikzcd}
\]
Hence, for $u\colon U\hook X$ an open immersion, we have factorizations
\[
\begin{tikzcd}
	\SH(U) \ar{r}{u_!} \ar[hook]{d} & \SH(X) \\
	\bigSH_\cdh(U) \ar{r}{u_!} & \bigSH_\cdh(X) \ar{u}
\end{tikzcd}
\quad
\begin{tikzcd}
	\SH(U) \ar{r}{u_*} \ar[hook]{d} & \SH(X) \\
	\bigSH_\cdh(U) \ar{r}{u_*} & \bigSH_\cdh(X)\rlap. \ar{u}
\end{tikzcd}
\]
Thus, to show that~\eqref{eqn:SH-transformation} is an equivalence, it suffices to show that the natural transformation
\begin{equation*}\label{eqn:SHcdh-transformation}
u_!t_! \to u_* t_!\colon \bigSH_\cdh(T) \to \bigSH_\cdh(X)
\end{equation*}
is an equivalence, which we do in Proposition~\ref{prop:!*} below. The following three lemmas are the heart of the proof.

\begin{lem}\label{lem:interval}
	Let $V$ be a valuation ring and $X$ a connected $V$-scheme. Then the image of $X\to\Spec V$ is an interval in the specialization poset.
\end{lem}

\begin{proof}
	This follows from \cite[Lemma 3.2.9]{cdarc}, which says that Milnor squares are preserved by pullback to a reduced scheme: if the fiber over $\p\subset V$ is empty, then $X_\red$ is the sum of its restrictions to $V_\p$ and $V/\p$.
\end{proof}

In the following lemmas, $\PSh_\emptyset\subset\PSh$ denotes the full subcategory of presheaves that send the initial object to the terminal object, and $\PSh_\Sigma\subset\PSh_\emptyset$ is the full subcategory of presheaves that transform finite sums into finite products.

\begin{lem}\label{lem:key}
	Let $V$ be a valuation ring of finite rank, let $X=\Spec V$, and let $T\xrightarrow t U\xrightarrow u X$ be open immersions with $T\neq U$.
	Let $\sC$ be a pointed $\infty$-category with finite products. Then the natural transformation
	\[
	u_!t_! \to u_* t_!\colon \PSh_\Sigma(\Sch_T^\fp,\sC) \to \PSh_\Sigma(\Sch_X^\fp,\sC)
	\]
	is an equivalence.
\end{lem}

\begin{proof}
	Since $V$ has finite rank, every scheme in $\Sch_X^\fp$ has finitely many generic points and in particular is a finite sum of connected schemes. It therefore suffices to show that
	\[
	(u_!t_!\sF)(Y) \simeq (u_* t_!\sF)(Y)
	\]
	for every connected $X$-scheme $Y$. We have
	\begin{align*}
	(u_!t_!\sF)(Y) &= \begin{cases}
		* & \text{if $Y_T\neq Y$,} \\
		\sF(Y_T) & \text{otherwise,}
	\end{cases}
	\\
	(u_*t_!\sF)(Y) &= \begin{cases}
		* & \text{if $Y_T\neq Y_U$,} \\
		\sF(Y_T) & \text{otherwise.}
	\end{cases}
	\end{align*}
	These obviously agree if $Y_T=\emptyset$, since $\sF(\emptyset)=*$, so we may assume $Y_T\neq\emptyset$.
	In this case, since $T\neq U$ and the image of $Y\to X$ is an interval (Lemma~\ref{lem:interval}), $Y_T\neq Y$ if and only if $Y_T\neq Y_U$.
\end{proof}

\begin{lem}\label{lem:cdh-cocontinuous}
	Let $u\colon U\hook X$ be an open immersion between qcqs schemes. Then the functors
	\begin{gather*}
	u_*\colon \PSh_\emptyset(\Sch_U^\fp) \to \PSh_\emptyset(\Sch_X^\fp) \\
	u_*\colon \PSh_\emptyset(\Sch_U^\fp)_* \to \PSh_\emptyset(\Sch_X^\fp)_*
	\end{gather*}
	preserve cdh-local equivalences and motivic equivalences (i.e., morphisms that become equivalences in $\bigH_\cdh$).
\end{lem}

\begin{proof}
	The functor $u_*\colon \PSh_\emptyset(\Sch_U^\fp) \to \PSh_\emptyset(\Sch_X^\fp)$ preserves colimits indexed by weakly contractible $\infty$-categories, since the inclusion $\PSh_\emptyset\subset\PSh$ does.
	If $L_\emptyset\colon \PSh(\Sch_U^\fp)\to\PSh_\emptyset(\Sch_U^\fp)$ is the left adjoint to the inclusion, then
	\[
	(L_\emptyset\sF)(Y)=\begin{cases} * & \text{if $Y=\emptyset$,} \\ \sF(Y) & \text{otherwise.} \end{cases}
	\]
	In particular, if $Y\in\Sch_U^\fp$ and $i\colon R\hook Y$ is a sieve, then $L_\emptyset(i)=i$ unless the sieve is empty, in which case $L_\emptyset(i)$ is the sieve on $Y$ generated by the empty scheme.
	The collection of cdh-local equivalences in $\PSh_\emptyset(\Sch_U^\fp)$ is therefore generated under 2-out-of-3 and colimits by nonempty cdh sieves, and the collection of motivic equivalences is similarly generated by cdh-local equivalences and $\A^1$-homotopy equivalences. The same collections are generated using only 2-out-of-3 and weakly contractible colimits, because the initial object of $\Fun(\Delta^1,\PSh_\emptyset(\Sch_U^\fp))$ is a cdh sieve and the colimit of any diagram $\sK\to \sC$ is the same as the colimit of an extension $\sK^\triangleleft\to\sC$ sending the cone point to an initial object.
	Since $u_*$ preserves $\A^1$-homotopic maps, it remains to show that for every nonempty cdh sieve $R\hook Y$ in $\Sch_U^\fp$, $u_*(R)\hook u_*(Y)$ is a cdh-local equivalence. Since it is a monomorphism, it suffices to check that it is surjective on stalks. If $A$ is a henselian valuation ring and $\Spec A\to u_*(Y)$ is a morphism, then $(\Spec A)_U$ is either empty or the spectrum of a henselian valuation ring \cite[Lemma 3.3.5]{cdarc}. In both cases, the map $(\Spec A)_U \to Y$ factors through $R$.
	
	The functor $u_*\colon \PSh_\emptyset(\Sch_U^\fp)_* \to \PSh_\emptyset(\Sch_X^\fp)_*$ preserves colimits, so as before it suffices to show that $u_*L_\emptyset(R_+) \hook u_* L_\emptyset(Y_+)$ is a cdh-local equivalence for every cdh sieve $R\hook Y$. 
	 Since it is a monomorphism, this can be checked on stalks as above.
\end{proof}

\begin{rem}\label{rem:etale-pushforward}
	If $u\colon U\to X$ is an étale morphism between qcqs schemes, the conclusions of Lemma~\ref{lem:cdh-cocontinuous} hold if one replaces $\PSh_\emptyset$ with $\PSh_\Sigma$. Indeed, if $V$ is a henselian valuation ring and $X\to \Spec V$ is a quasi-compact étale morphism, then $X$ is the spectrum of a finite product of henselian valuation rings.
\end{rem}

\begin{rem}
	Lemma~\ref{lem:cdh-cocontinuous} (but not Remark~\ref{rem:etale-pushforward}) also holds for the rh topology, whose points are valuation rings.
\end{rem}

\begin{prop}\label{prop:!*}
	Let $V$ be a valuation ring of finite rank, let $X=\Spec V$, and let $T\xrightarrow t U\xrightarrow u X$ be open immersions with $T\neq U$. Then the natural transformations
	\begin{gather*}
	u_!t_! \to u_* t_!\colon \bigH_\cdh(T)_* \to \bigH_\cdh(X)_* \\
	u_!t_! \to u_* t_!\colon \bigSH_\cdh(T) \to \bigSH_\cdh(X)
	\end{gather*}
	are equivalences.
\end{prop}

\begin{proof}
	The first equivalence follows directly from Lemmas~\ref{lem:key} and \ref{lem:cdh-cocontinuous}. 
	The functors $u_!$ and $u_*$ extend to functors between the $\infty$-categories of $\P^1$-prespectra, which are computed levelwise.
	The second equivalence follows from the first since both $u_!$ and $u_*$ commute with spectrification (the former because $u^*$ preserves $\P^1$-spectra among $\P^1$-prespectra, and the latter because $u_*$ commutes with $\P^1$-loops and filtered colimits).
\end{proof}

This completes the proof of Theorem~\ref{thm:v-excision}, hence of Theorem~\ref{thm:main}.

\begin{rem}
	Let $S$ be a scheme and $E\in \mathrm{Alg}(\SH(S))$ a motivic ring spectrum over $S$. It follows formally from Theorem~\ref{thm:main} that the presheaf of $\infty$-categories $\mathrm{Mod}_E(\SH(-))\colon \Sch_S^\op \to\Cat_\infty$ satisfies Milnor excision. For example, the $\infty$-category of Beilinson motives \cite[\sectsign 14.2]{CD}, the $\infty$-category of Spitzweck motives \cite[Chapter 9]{SpitzweckHZ}, and the $\infty$-category of motivic spectra with finite syntomic transfers \cite[\sectsign 4.1]{deloop3} satisfy Milnor excision.
\end{rem}

\section{Milnor excision for étale motives}

We conclude this article with a proof of Milnor excision for Ayoub's étale motives.
For $X$ a scheme and $\Lambda$ a commutative ring, let $\DA^\et(X,\Lambda)$ be the $\infty$-category of étale motives constructed in \cite[\sectsign 3]{AyoubEtale}, and let $\DA^\et_\ct(X,\Lambda)\subset \DA^\et(X,\Lambda)$ be the subcategory of constructible objects (in the sense of \cite[Définition 8.1]{AyoubEtale})\footnote{A more natural subcategory is that of locally constructible objects in the sense of \cite[Definition 6.3.1]{CDetale}. We will only use $\DA^\et_\ct$ in cases where the two subcategories coincide.}.
Let $\mathbf D^\et(X,\Lambda)$ be the derived $\infty$-category of the abelian category of étale sheaves of $\Lambda$-modules on $X$ (equivalently, the $\infty$-category of étale hypersheaves of $\Lambda$-module spectra), and let $\mathbf D^\et_\ct(X,\Lambda)\subset\mathbf D^\et(X,\Lambda)$ be the subcategory of constructible objects (in the sense of \cite[Definition 6.3.1]{proetale}).
We denote by $H\Lambda\in \SH(X)$ Spitzweck's motivic cohomology spectrum with coefficients in $\Lambda$ \cite[\sectsign 4]{Hoyois:2018aa}.

We shall say that $X$ is \emph{$\Lambda$-finite} if it has finite Krull dimension and
\[\sup_{x\in X, p\notin\Lambda^\times}\cd_p(\kappa(x))<\infty,\]
where $\cd_p(k)$ is the mod $p$ Galois cohomological dimension of a field $k$.
 If $X$ is quasi-compact and $\Lambda$-finite and $Y\to X$ is of finite type, then $Y$ is also $\Lambda$-finite. 
We shall say that $X$ is \emph{étale-locally $\Lambda$-finite} if it admits an étale covering by $\Lambda$-finite schemes. A quasi-compact scheme $X$ is étale-locally $\Lambda$-finite if and only if the schemes $X[\tfrac 12,\zeta_4]$ and $X[\tfrac 13,\zeta_6]$ are $\Lambda$-finite, and also if and only if $X$ has finite Krull dimension and $\sup_{x\in X,p\notin\Lambda^\times}\vcd_p(\kappa(x))<\infty$. Every scheme essentially of finite type over $\Z$ is étale-locally $\Lambda$-finite, and the collection of étale-locally $\Lambda$-finite schemes is closed under Milnor pushouts.

\begin{lem}\label{lem:Det}
	Let $X$ be a scheme and $\Lambda$ a commutative ring.
	\begin{enumerate}
		\item If $X$ is qcqs and $\Lambda$-finite, then 
		\[
		\mathbf D^\et(X,\Lambda) \simeq \Ind(\mathbf D^\et_\ct(X,\Lambda))\quad\text{and}\quad \DA^\et(X,\Lambda) \simeq \Ind(\DA^\et_\ct(X,\Lambda)).
		\]
		\item If $X$ is the limit of a cofiltered diagram of schemes $X_i$ with affine transition morphisms and if $X$ and $X_i$ are étale-locally $\Lambda$-finite, then 
		\[
		\mathbf D^\et(X,\Lambda)\simeq \lim_i\mathbf D^\et(X_i,\Lambda)\quad\text{and}\quad \DA^\et(X,\Lambda)\simeq \lim_i\DA^\et(X_i,\Lambda).
		\]
		\item If $f\colon Y\to X$ is a qcqs morphism and $X$ and $Y$ are étale-locally $\Lambda$-finite, then
		\[
		f_*\colon \mathbf D^\et(Y,\Lambda)\to \mathbf D^\et(X,\Lambda)\quad\text{and}\quad f_*\colon \DA^\et(Y,\Lambda)\to \DA^\et(X,\Lambda)
		\]
		 preserve colimits.
	\end{enumerate}
\end{lem}

\begin{proof}
	(i) Let $d = \dim(X)+\sup_{x\in X,p\notin\Lambda^\times}\cd_p(\kappa(x))+1$.
	For any qcqs étale $X$-scheme $U$ and any étale sheaf of $\Lambda$-modules $\sF$ on $U$, we have $H^n_\et(U,\sF)=0$ for $n>d$. Indeed, this follows from \cite[Corollary 3.29]{clausen-mathew}, noting that the $p$-local Galois cohomological dimension of a field $k$ is at most $\cd_p(k)+1$.
	The result for $\mathbf D^\et$ is now \cite[Proposition 6.4.8]{proetale}, and the analogue for $\DA^\et$ is an easy consequence (cf.\ \cite[Proposition 3.19]{AyoubEtale}).
	
	(ii) By Zariski descent, we may assume that $X_i$ and hence $X$ are qcqs. 
	Using descent along the étale covering $\{\Spec \Z[\tfrac 12,\zeta_4],\Spec \Z[\tfrac 13,\zeta_6]\}$ of $\Spec \Z$, we may further assume that $X$ and $X_i$ are $\Lambda$-finite.
	 The result is then a formal consequence of (i), cf.\ \cite[Proposition~3.20]{AyoubEtale}.
	 
	 (iii) By étale descent, we can assume that $X$ and $Y$ are qcqs and $\Lambda$-finite. Then the result follows immediately from (i).
\end{proof}

\begin{rem}
	For $\mathbf D^\et(X,\Lambda)$ to be compactly generated, it suffices that $X$ be qcqs and étale-locally $\Lambda$-finite. However, this does not suffice for the conclusion of Lemma~\ref{lem:Det}(i), as the case $X=\Spec\mathbf R$ and $\Lambda=\Z/2$ shows: there the unit object is constructible but not compact.
\end{rem}

\begin{lem}\label{lem:DA}
	Let $X$ be a scheme and $\Lambda$ a commutative ring.
	\begin{enumerate}
		\item If $\Lambda$ is a $\Q$-algebra and $X$ is locally of finite Krull dimension, then $\DA^\et(X,\Lambda)\simeq \Mod_{H\Lambda}(\SH(X))$.
		\item If $\Lambda$ is a $\Z/n$-algebra for some integer $n$ invertible on $X$ and if $X_x^\mathrm{sh}$ is $\Lambda$-finite for every geometric point $x$ of $X$, then $\DA^\et(X,\Lambda)\simeq \mathbf D^\et(X,\Lambda)$.
		\item If $p$ is a prime that is locally nilpotent on $X$, then $\DA^\et(X,\Lambda)\simeq \DA^\et(X,\Lambda[\tfrac 1p])$.
	\end{enumerate}
\end{lem}

\begin{proof}
	(i) We may assume $\Lambda=\Q$, as both sides are obtained from this case by taking $\Lambda$-modules. By construction, Spitzweck's $H\Q$ is the Beilinson motivic cohomology spectrum of Cisinski and Déglise. 
	Using Zariski descent and Lemma~\ref{lem:Det}(ii), we can assume $X$ noetherian of finite Krull dimension. In this case the result is precisely \cite[Theorem 16.2.18]{CD}.
	
	(ii) If $X$ is of finite type over $\Z$, this follows from \cite[Théorème 4.1]{AyoubEtale}. By Lemma~\ref{lem:Det}(ii), the conclusion holds whenever $X$ is qcqs and étale-locally $\Lambda$-finite. The general case (which we will not use) follows from this case as in the second half of the proof of \cite[Théorème 4.1]{AyoubEtale}.
	
	(iii) By nilinvariance, we can assume that $X$ is an $\mathbf F_p$-scheme. Then the result follows from the Artin–Schreier exact sequence, see \cite[Lemma 3.10]{AyoubEtale}.
\end{proof}

\begin{rem}
	By a theorem of Gabber \cite[Exposé XVIII\textsubscript{A}, Corollaire 1.2]{Gabber}, the assumption on $X$ in Lemma~\ref{lem:DA}(ii) holds whenever $X$ is locally noetherian. However, it does not hold in general.
	Indeed, if $\Gamma$ is any totally ordered abelian group, there exists a strictly henselian valuation ring $R$ with value group $\Gamma$ (for example, the strict henselization of the ring of Hahn series $k[[t^{\Gamma_{\geq 0}}]]$ over a field $k$). If $l$ is a prime invertible in $R$, the $l$-cohomological dimension of $\Frac(R)$ is then equal to the rank of $\Gamma/l$ \cite[Lemma 1.2]{Chipchakov}, which can be infinite even if $R$ has rank $1$ (i.e., even if $\Gamma\subset\mathbf R$).
\end{rem}

\begin{lem}\label{lem:fracture}
	Let $\sC$ be a cocomplete prestable $\infty$-category, which is either compactly generated or $\Z$-linear. If $X\in\sC$ is such that $X\otimes\Q=0$ and $X/p=0$ for every prime $p$, then $X=0$.
\end{lem}

\begin{proof}
	Suppose first that $\sC$ is compactly generated. Since $\sC$ is additive, it is enough to show that $[K,X]=0$ for every compact object $K\in\sC$. By compactness, we have $[K,X]\otimes\Q\simeq [K,X\otimes \Q]=0$, so $[K,X]$ is torsion. From the fiber sequence $X\to X\to X/p$ (using the prestability of $\sC$), we see that $\mathrm{Tor}([K,X],\Z/p)$ is a quotient of $[K,\Omega(X/p)]=0$ , so $[K,X]$ is torsionfree. Hence, $[K,X]=0$.
	Suppose now that $\sC$ is $\Z$-linear. Since $\Q/\Z$ is the filtered colimit of the $\Z$-modules $\Z/n$, we have $X\otimes_\Z \Q/\Z=0$. Using the cofiber sequence $\Z\to \Q\to \Q/\Z$ and the prestability of $\sC$, we deduce that $X=0$.
\end{proof}

\begin{thm}\label{thm:DA}
	Let $\Lambda$ be a commutative ring. The presheaf of $\infty$-categories $\DA^\et(-,\Lambda)$ satisfies Milnor excision on the category of étale-locally $\Lambda$-finite schemes.
\end{thm}

\begin{proof}
	For any morphism $\Lambda \to \Lambda'$, we have $\DA^\et(-,\Lambda')\simeq \Mod_{\Lambda'}(\DA^\et(-,\Lambda))$.
	We can thus assume that $\Lambda$ is a localization of $\Z$. Moreover, by Zariski descent and descent along the étale covering $\{\Spec \Z[\tfrac 12,\zeta_4],\Spec \Z[\tfrac 13,\zeta_6]\}$ of $\Spec \Z$, it suffices to consider $\Lambda$-finite qcqs schemes. Consider a Milnor square of such schemes
	\[
	\begin{tikzcd}
		W \ar{d}[swap]{g} \ar{r}{k} & Y \ar{d}{f} \\
		Z \ar{r}{i} & X\rlap.
	\end{tikzcd}
	\]
	By localization, the functor $(f^*,i^*)\colon \DA^\et(X,\Lambda) \to \DA^\et(Y,\Lambda)\times \DA^\et(Z,\Lambda)$ is conservative. As explained in Section~\ref{sec:six-op}, $\DA^\et(-,\Lambda)$ takes this square to a cartesian square if and only if certain morphisms in $\DA^\et(Y,\Lambda)$ and $\DA^\et(Z,\Lambda)$ are equivalences. By Lemma~\ref{lem:fracture}, this can be checked rationally and modulo $p$ for every prime $p$. Rationally, we have $\DA^\et(-,\Q)\simeq \Mod_{H\Q}(\SH(-))$ by Lemma~\ref{lem:DA}(i), and the latter satisfies Milnor excision by Theorem~\ref{thm:main}. By Lemma~\ref{lem:Det}(iii), the morphisms witnessing Milnor excision for $\DA^\et(-,\Q)$ are the rationalizations of the ones for $\DA^\et(-,\Lambda)$, hence the latter are rational equivalences. Modulo $p$, we have $\DA^\et(-,\Z/p)\simeq \mathbf D^\et(-[\tfrac 1p],\Z/p)$ by Lemma~\ref{lem:DA}(ii,iii) and localization. By Lemma~\ref{lem:Det}(i) and Lemma~\ref{lem:cg-pullback}(ii), it remains to show that $\mathbf D^\et_\ct(-[\tfrac 1p],\Z/p)$ satisfies Milnor excision. This is true (on all schemes) by \cite[Theorem 5.14]{arc} and \cite[Corollary 3.2.12]{cdarc}.
\end{proof}
 
\subsection*{Acknowledgments}
We are grateful to Joseph Ayoub for asking about Theorem~\ref{thm:DA} and for discussions about the proof, and to Denis-Charles Cisinski for further comments.

We thank the Isaac Newton Institute for Mathematical Sciences for support and hospitality during the program ``K-theory, algebraic cycles and motivic homotopy theory'' where this paper was written. This work was partially supported by EPSRC grant number EP/R014604/1.

\bibliographystyle{alphamod}

\let\mathbb=\mathbf

{\small
\bibliography{excise}

\newcommand{\etalchar}[1]{$^{#1}$}
\providecommand{\bysame}{\leavevmode\hbox to3em{\hrulefill}\thinspace}
\providecommand{\MR}{\relax\ifhmode\unskip\space\fi MR }
\providecommand{\MRhref}[2]{%
  \href{http://www.ams.org/mathscinet-getitem?mr=#1}{#2}
}
\providecommand{\href}[2]{#2}
\begin{thebibliography}{EHK{\etalchar{+}}20}
\providecommand{\url}[1]{\href{#1}{{\def~{\textasciitilde}\tt #1}}}

\bibitem[Ayo14]{AyoubEtale}
J.~Ayoub, \emph{La r{\'e}alisation {\'e}tale et les op{\'e}rations de
  Grothendieck}, Ann. Sci. {\'E}c. Norm. Sup{\'e}r. \textbf{47} (2014), no.~1,
  pp.~1--145

\bibitem[BM20]{arc}
B.~Bhatt and A.~Mathew, \emph{The arc topology}, 2020,
  \href{http://arxiv.org/abs/arXiv:1807.04725v4}{{\sf arXiv:1807.04725v4}}

\bibitem[BS15]{proetale}
B.~Bhatt and P.~Scholze, \emph{The pro-{\'e}tale topology for schemes},
  Ast{\'e}risque \textbf{369} (2015), pp.~99--201

\bibitem[CD16]{CDetale}
D.-C. Cisinski and F.~D{\'e}glise, \emph{{\'E}tale motives}, Compos. Math.
  \textbf{152} (2016), no.~3, pp.~556--666

\bibitem[CD19]{CD}
\bysame, \emph{Triangulated categories of mixed motives}, Springer Monographs
  in Mathematics, Springer, 2019, preprint
  \href{http://arxiv.org/abs/0912.2110}{{\sf arXiv:0912.2110}}

\bibitem[Chi02]{Chipchakov}
I.~D. Chipchakov, \emph{On the Galois cohomological dimensions of stable fields
  with Henselian valuations}, Comm. Algebra \textbf{30} (2002), no.~4,
  pp.~1549--1574

\bibitem[Cis13]{Cisinski}
D.-C. Cisinski, \emph{Descente par {\'e}clatements en $K$-th{\'e}orie
  invariante par homotopie}, Ann. Math. \textbf{177} (2013), no.~2,
  pp.~425--448

\bibitem[CM21]{clausen-mathew}
D.~Clausen and A.~Mathew, \emph{Hyperdescent and {\'e}tale K-theory}, Invent.
  Math. \textbf{225} (2021), pp.~981--1076, preprint
  \href{http://arxiv.org/abs/1905.06611}{{\sf arXiv:1905.06611}}

\bibitem[EHIK21]{cdarc}
E.~Elmanto, M.~Hoyois, R.~Iwasa, and S.~Kelly, \emph{Cdh descent, cdarc
  descent, and Milnor excision}, Math. Ann. \textbf{379} (2021),
  pp.~1011--1045, preprint \href{http://arxiv.org/abs/2002.11647}{{\sf
  arXiv:2002.11647}}

\bibitem[EHK{\etalchar{+}}20]{deloop3}
E.~Elmanto, M.~Hoyois, A.~A. Khan, V.~Sosnilo, and M.~Yakerson, \emph{Modules
  over algebraic cobordism}, Forum Math. Pi \textbf{8} (2020), p.~e14

\bibitem[Hoy14]{HoyoisGLV}
M.~Hoyois, \emph{A quadratic refinement of the
  {G}rothendieck--{L}efschetz--{V}erdier trace formula}, Algebr. Geom. Topol.
  \textbf{14} (2014), no.~6, pp.~3603--3658

\bibitem[Hoy17]{hoyois-sixops}
M.~Hoyois, \emph{The six operations in equivariant motivic homotopy theory},
  Adv. Math. \textbf{305} (2017), pp.~197--279

\bibitem[Hoy21]{Hoyois:2018aa}
M.~Hoyois, \emph{The localization theorem for framed motivic spaces}, Compos.
  Math. \textbf{157} (2021), no.~1, pp.~1--11, preprint
  \href{http://arxiv.org/abs/1807.04253}{{\sf arXiv:1807.04253}}

\bibitem[ILO13]{Gabber}
L.~Illusie, Y.~Laszlo, and F.~Orgogozo, \emph{Travaux de Gabber sur
  l'uniformisation locale et la cohomologie {\'e}tale des sch{\'e}mas
  quasi-excellents}, Ast{\'e}risque \textbf{363} (2013)

\bibitem[Kel19]{shane-better}
S.~Kelly, \emph{A better comparison of cdh- and ldh-cohomologies}, Nagoya Math.
  J. \textbf{236} (2019), pp.~183--213, preprint
  \href{http://arxiv.org/abs/1807.00158}{{\sf arXiv:1807.00158}}

\bibitem[Kha19]{shcdh}
A.~A. Khan, \emph{The cdh-local motivic homotopy category (after Cisinski)},
  2019,  \url{https://www.preschema.com/documents/shcdh.pdf}

\bibitem[KM21]{kelly-morrow}
S.~Kelly and M.~Morrow, \emph{$K$-theory of valuation rings}, Compos. Math.
  \textbf{157} (2021), no.~6, pp.~1121--1142, preprint
  \href{http://arxiv.org/abs/1810.12203}{{\sf arXiv:1810.12203}}

\bibitem[KST17]{KST}
M.~Kerz, F.~Strunk, and G.~Tamme, \emph{Algebraic $K$-theory and descent for
  blow-ups}, Invent. Math. \textbf{211} (2017), no.~2

\bibitem[LT19]{land-tamme}
M.~Land and G.~Tamme, \emph{On the $K$-theory of pullbacks}, Ann. Math.
  \textbf{190} (2019), no.~3, pp.~877--930, preprint
  \href{http://arxiv.org/abs/1808.05559}{{\sf arXiv:1808.05559}}

\bibitem[Lur17a]{HA}
J.~Lurie, \emph{Higher Algebra}, September 2017,
  \url{http://www.math.harvard.edu/~lurie/papers/HA.pdf}

\bibitem[Lur17b]{HTT}
\bysame, \emph{Higher Topos Theory}, April 2017,
  \url{http://www.math.harvard.edu/~lurie/papers/HTT.pdf}

\bibitem[Mil71]{milnor-ktheory}
J.~Milnor, \emph{Introduction to algebraic {$K$}-theory}, Annals of Mathematics
  Studies, No. 72

\bibitem[Spi18]{SpitzweckHZ}
M.~Spitzweck, \emph{A commutative {$\Bbb P^1$}-spectrum representing motivic
  cohomology over {D}edekind domains}, M\'{e}m. Soc. Math. Fr. (N.S.) (2018),
  no.~157, p.~110

\bibitem[Sus95]{suslin-excision}
A.~A. Suslin, \emph{Excision in integer algebraic {$K$}-theory}, Trudy Mat.
  Inst. Steklov. \textbf{208} (1995), pp.~290--317, Teor. Chisel, Algebra i
  Algebr. Geom. (Russian)

\bibitem[SW92]{suslin-wodzicki}
A.~Suslin and M.~Wodzicki, \emph{Excision in algebraic $K$-theory}, Ann. Math.
  \textbf{136} (1992), no.~1, pp.~51--122

\bibitem[Swa71]{swan}
R.~G. Swan, \emph{Excision in algebraic {$K$}-theory}, J. Pure Appl. Algebra
  \textbf{1} (1971), no.~3, pp.~221--252

\bibitem[Swa80]{swan-seminormality}
\bysame, \emph{On seminormality}, J. Algebra \textbf{67} (1980), no.~1,
  pp.~210--229

\bibitem[Wei89]{Weibel}
C.~A. Weibel, \emph{Homotopy algebraic $K$-theory}, Algebraic K-Theory and
  Number Theory, Contemp. Math., vol.~83, AMS, 1989, pp.~461--488

\end{thebibliography}
}

\parskip 0pt

\end{document}